\documentclass{amsart}
\usepackage{amsmath}
\usepackage{amssymb}
\usepackage{amsthm}
\usepackage{enumerate}
\usepackage[dvips]{graphicx}
\usepackage{textcomp}
\theoremstyle{definition}
\newtheorem{definition}{Definition}[section]
\theoremstyle{plain}
\newtheorem{lemma}[definition]{Lemma}
\newtheorem{theorem}[definition]{Theorem}
\newtheorem{proposition}[definition]{Proposition}
\newtheorem{corollary}[definition]{Corollary}
\theoremstyle{remark}
\newtheorem{remark}[definition]{Remark}
\newtheorem{notation}[definition]{Notation}

\newcommand{\myint}{\operatorname{int}}
\newcommand{\myfdim}{\operatorname{fdim}}

\makeatletter
\@namedef{subjclassname@2010}{
  \textup{2010} Mathematics Subject Classification}
\makeatother

\begin{document}

\title[Locally o-minimal structures with tame properties]{Locally o-minimal structures with tame topological properties}
\author[M. Fujita]{Masato Fujita}
\address{Department of Liberal Arts,
Japan Coast Guard Academy,
5-1 Wakaba-cho, Kure, Hiroshima 737-8512, Japan}
\email{fujita.masato.p34@kyoto-u.jp}

\begin{abstract}
We consider locally o-minimal structures possessing tame topological properties shared by models of DCTC and uniformly locally o-minimal expansions of the second kind of densely linearly ordered abelian groups.
We derive basic properties of dimension of a set definable in the structures including the
addition property, which is the dimension equality for definable maps whose fibers are equi-dimensional.
A decomposition theorem into quasi-special submanifolds is also demonstrated.
\end{abstract}



\maketitle

\section{Introduction}\label{sec:intro}
An o-minimal structure enjoys many tame topological properties such as monotonicity and definable cell decomposition \cite{vdD}.
A locally o-minimal structure was first introduced in \cite{TV} as a local counterpart of an o-minimal structure.
In spite of its similarity to an o-minimal structure in its definition, a locally o-minimal structure does not enjoy the localized tame properties enjoyed by o-minimal structures such as the local monotonicity theorem and the local definable cell decomposition theorem.
Lack of tame topological properties prevents us to establish a tame dimension theory for sets definable in the structures.
We expect that discrete definable set is of dimension zero.
We also hope that the projection image of a definable set is of dimension not greater than the dimension of the original set.
However, the projection image of a discrete definable set is not necessarily discrete in some locally o-minimal structure as in \cite[Example 12]{KTTT}.

We can recover tame topological properties if we employ additional assumptions on locally o-minimal structures.
We can also establish a tame dimension theory using such tame topological properties.

For instance, the author proposed uniformly locally o-minimal structures of the second kind in \cite{Fuji}.
Local definable cell decomposition theorem \cite[Theorem 4.2]{Fuji} holds true when they are definably complete.
We can derive several natural dimension formulae \cite[Section 5]{Fuji} and \cite[Theorem 1.1, Corollary 1.2]{Fuji2} for a definably complete uniformly locally o-minimal expansion of the second kind of a densely linearly ordered abelian group using the tame topological properties.
A definably complete uniformly locally o-minimal expansion of the second kind of a densely linearly ordered abelian group is called a DCULOAS structure in this paper.

Another example is a model of DCTC.
Schoutens tried to figure out the common features of the models of the theory of all o-minimal structures in his challenging work \cite{S}.
A model of DCTC was introduced in it.
It enjoys tame topological and dimensional properties as partially given in \cite{S} and also demonstrated in this paper.

The purpose of this paper is to develop dimension formulae when locally o-minimal structures are definably complete and enjoy the tame topological property given in the following definition.
The previous two examples possess this property.
\begin{definition}\label{def:tame_top}
Consider a locally o-minimal structure.
We consider the following property on it.
\begin{enumerate}
\item[(a)]  The image of a nonempty definable discrete set under a coordinate projection is again discrete.
\end{enumerate}
\end{definition}

The following formulae on dimensions are demonstrated in this paper under the assumption that definably complete locally o-minimal structures enjoy the property (a) in Definition \ref{def:tame_top}.
\begin{enumerate}
\item[(1)] The inequality on the dimensions of the domain of definition and the image of a definable map (Theorem \ref{thm:dim}(5));
\item[(2)] The inequality on the dimension of the set of points at which a definable function is discontinuous (Theorem \ref{thm:dim}(6));
\item[(3)] The inequality on the dimensions of a definable set and its frontier (Theorem \ref{thm:dim}(7));
\item[(4)] Addition property. The dimension equality for definable maps whose fibers are equi-dimensional (Theorem \ref{thm:addition}).
\end{enumerate}

In o-minimal structures, definable sets are partitioned into finite number of nicely shaped definable subsets called cells \cite[Chapter 3 (2.11)]{vdD}.
Partitions into finite cells are unavailable in locally o-minimal structures.
We provide alternative partitions into finite number of another nicely shaped definable subsets called quasi-special submanifolds.
The definition of quasi-special submanifolds is found in Definition \ref{def:quasi-special}.
Partitions into quasi-special submanifolds are available in locally o-minimal structures enjoying the property (a) in Definition \ref{def:tame_top} (Theorem \ref{thm:decomposition} and Theorem \ref{thm:frontier_condition}).
Quasi-special submanifolds only satisfy looser conditions than special manifolds defined in \cite{M2,T}.
Decomposition theorems into special submanifolds hold true for locally o-minimal expansions of fields \cite{F} and d-minimal expansions of the real field \cite{M2,T}.
Unlike special submanifolds, our partitions into quasi-special submanifolds are available without assuming that the structure is an expansion of an ordered field.
It is an advantage of our decomposition theorem.

A DCULOAS structure and a model of DCTC possess the property (a) in Definition \ref{def:tame_top}.
Therefore, the above dimension formulae and the decomposition theorem into quasi-special submanifolds also hold true for them.
Some of the assertions were presented in the previous studies.
In the case of a DCULOAS structure, the dimension inequalities (1) through (3) were demonstrated in \cite{Fuji, Fuji2}.
The addition property (4) and the decomposition theorem first appear in this paper. 
As to a model of DCTC, the inequalities in the planar case were proved in \cite{S}. 
The author could not find the dimension formulae for higher dimensions in the previous studies.

This paper is organized as follows.
We first derive basic topological properties of the structure in Section \ref{sec:prel}.
Section \ref{sec:basic} is devoted for the derivation of the dimension formulae (1) through (4).
We also prove the decomposition theorem into quasi-special submanifolds in Section \ref{sec:decomposition}.

We introduce the terms and notations used in this paper.
The term `definable' means `definable in the given structure with parameters' in this paper.
For any set $X \subset M^{m+n}$ definable in a structure $\mathcal M=(M,\ldots)$ and for any $x \in M^m$, the notation $X_x$ denotes the fiber defined as $\{y \in M^n\;|\; (x,y) \in X\}$ unless another definition is explicitly given.
For a linearly ordered structure $\mathcal M=(M,<,\ldots)$, an open interval is a definable set of the form $\{x \in R\;|\; a < x < b\}$ for some $a,b \in M$.
It is denoted by $(a,b)$ in this paper.
An open box in $M^n$ is the direct product of $n$ open intervals.
Let $A$ be a subset of a topological space.
The notations $\myint(A)$ and $\overline{A}$ denote the interior and the closure of the set $A$, respectively.

\section{Tame topological properties}\label{sec:prel}
\subsection{Basic lemmas}
We first review the definitions of local o-minimality and definably completeness.
\begin{definition}[\cite{TV}]
A densely linearly ordered structure without endpoints $\mathcal M=(M,<,\ldots)$ is \textit{locally o-minimal} if, for every definable subset $X$ of $M$ and for every point $a\in M$, there exists an open interval $I$ containing the point $a$ such that $X \cap I$ is  a finite union of points and open intervals.
\end{definition}

\begin{definition}[\cite{M}]
An expansion of a densely linearly ordered set without endpoints $\mathcal M=(M,<,\ldots)$ is \textit{definably complete} if any definable subset $X$ of $M$ has the supremum and  infimum in $M \cup \{\pm \infty\}$.
\end{definition}

We give an equivalence condition for a definably complete structure being locally o-minimal. 

\begin{lemma}\label{lem:local}
Consider a definably complete structure $\mathcal M=(M,<,\ldots)$.
The following are equivalent:
\begin{enumerate}
\item[(1)] The structure $\mathcal M$ is a locally o-minimal structure.
\item[(2)] Any definable set in $M$ has a nonempty interior or it is closed and discrete. 
\end{enumerate}
\end{lemma}
\begin{proof}
The implication (1) $\Rightarrow$ (2) is obvious by the definition of local o-miniality.
We demonstrate the opposite implication.
Let $X$ be a definable subset in $M$.
Consider the boundary $Y=\overline{X} \setminus \myint(X)$.
Let $J$ be an arbitrary open interval in $M$.
We have $Y \cap J = \emptyset$ if and only if $J \subset \myint(X)$ or $J \subset M \setminus \overline{X}$ by \cite[Corollary 1.5]{M}.
For any arbitrary point $a \in M$, there exists an open interval $I$ containing the point $a$ such that $I \cap Y$ is an empty set or a singleton $\{a\}$ because $Y$ is closed and discrete by the assumption.
The open intervals $\{x \in I\;|\; x>a\}$ and $\{x \in I\;|\; x<a\}$ are contained in $\myint(X)$ or $M \setminus \overline{X}$.
Hence, $I \cap X$ is a finite union of points and open intervals.
We have demonstrated that the structure $\mathcal M$ is locally o-minimal.
\end{proof}

We introduce two consequences of the property (a) in Definition \ref{def:tame_top}.

\begin{lemma}\label{lem:key0}
Consider a definably complete locally o-minimal structure with the property (a) in Definition \ref{def:tame_top}.
A definable discrete set is closed.
\end{lemma}
\begin{proof}
Let $\mathcal M=(M,<,\ldots)$ be the structure in consideration.
Let $X$ be a nonempty discrete definable subset of $M^n$.
Let $\pi_k:M^n \rightarrow M$ be the coordinate projection onto the $k$-th coordinate for all $1 \leq k \leq n$.
The images $\pi_k(X)$ are discrete by the property (a).
They are closed by Lemma \ref{lem:local}.
Let $x$ be an accumulation point of $X$.
We have $\pi_k(x) \in \pi_k(X)$ for all $1 \leq k \leq n$ because $\pi_k(x)$ are accumulation points of $\pi_k(X)$ and $\pi_k(X)$ are closed.
We can take open intervals $I_k$ so that $\pi_k(X) \cap I_k =\{\pi_k(x)\}$ because $\pi_k(X)$ are discrete.
It implies that $X \cap (I_1 \times \cdots \times I_n)$ consists of at most one point $x$ because $X \cap (I_1 \times \cdots \times I_n) \subseteq \prod_{k=1}^n \pi_k(X) \cap I_k = \{x\}$.
It means that $x \in X$ because $x$ is an accumulation point of $X$.
\end{proof}

\begin{lemma}\label{lem:aaa}
Consider a definably complete locally o-minimal structure $\mathcal M=(M,<,\ldots)$ with the property (a) in Definition \ref{def:tame_top}.
Let $f:X \rightarrow M$ be a definable map.
If the image $f(X)$ and all fibers of $f$ are discrete, then so is $X$. 
\end{lemma}
\begin{proof}
We first reduce to the case in which $f$ is the restriction of a coordinate projection.
Let $X$ be a definable subset of $M^n$ and $\pi:M^{n+1} \rightarrow M$ be the coordinate projection onto the last coordinate.
Consider the graph $\Gamma(f)$ of the definable map $f$.
The image $\pi(\Gamma(f))=f(X)$ and all the fibers $\Gamma(f) \cap \pi^{-1}(x)$ are discrete by the assumption.
If the graph $\Gamma(f)$ is discrete, the definable set $X$ is also discrete by the property (a) because $X$ is the projection image of the discrete set  $\Gamma(f)$.
We have reduced to the case in which $f$ is the restriction of the coordinate projection onto the last coordinate $\pi:M^{n+1} \rightarrow M$ to a definable subset $Y$ of $M^{n+1}$.

Take an arbitrary point $x \in Y$.
Since $\pi(Y)$ is discrete by the assumption, we can take an open interval $I$ containing the point $\pi(x)$ such that $\pi(Y) \cap I$ is a singleton.
Since the fiber $\pi^{-1}(\pi(x)) \cap Y$ is discrete, there exists an open box $B$ containing the point $x$ such that $Y \cap (B \times \{\pi(x)\})$ is a singleton.
The open box $B \times I$ contains the point $x$ and the intersection of $Y$ with $B \cap I$ is a singleton.
We have demonstrated that $Y$ is discrete.
\end{proof}

\subsection{Tame topological properties}

We defined the property (a) in Definition \ref{def:tame_top}.
We also consider the following topological properties in this paper.
\begin{definition}\label{def:tame_top2}
Consider a locally o-minimal structure $\mathcal M=(M,<,\ldots)$.
We consider the following properties on $\mathcal M$.
\begin{enumerate}
\item[(b)] Let $X_1$ and $X_2$ be definable subsets of $M^m$.
Set $X=X_1 \cup X_2$.
Assume that $X$ has a nonempty interior.
At least one of $X_1$ and $X_2$ has a nonempty interior.
\item[(c)] Let $A$ be a definable subset of $M^m$ with a nonempty interior and $f:A \rightarrow M^n$ be a definable map.
There exists a definable open subset $U$ of $M^m$ contained in $A$ such that the restriction of $f$ to $U$ is continuous.
\item[(d)] Let $X$ be a definable subset of $M^n$ and $\pi: M^n \rightarrow M^d$ be a coordinate projection such that the the fibers $X \cap \pi^{-1}(x)$ are discrete for all $x \in \pi(X)$.
Then, there exists a definable map $\tau:\pi(X) \rightarrow X$ such that $\pi(\tau(x))=x$ for all $x \in \pi(X)$.
\end{enumerate}
These properties and the property (a) in Definition \ref{def:tame_top} are not independent.
For definably complete locally o-minimal structures, the property (a) is equivalent to the properties (c) and (d).
The property (c) implies the property (b).
They are demonstrated in Theorem \ref{thm:dependence}.
\end{definition}

We introduce the following notations for simplicity.
\begin{notation}
Consider a locally o-minimal structure $\mathcal M=(M,<,\ldots)$.
A definable function $f:X \rightarrow M \cup \{\infty\}$ denotes a pair of disjoint definable subsets $X_o$ and $X_\infty$ with $X=X_o \cup X_\infty$ and a definable function defined on $X_o$. 
We consider that the function $f$ is constantly $\infty$ on $X_\infty$.
The function $f:X \rightarrow M \cup \{\infty\}$ is called continuous if $X=X_\infty$ or $X=X_o$ and the function $f$ is continuous.
If the structure $\mathcal M$ enjoys the properties (b) and (c) in Definition \ref{def:tame_top2}, the restriction of $f$ to some definable open set is continuous when the domain of definition $X$ has a nonempty interior.
We define $g:X \rightarrow M \cup \{-\infty\}$ similarly.
\end{notation}

The following lemma is a consequence of the properties (b) and (c).

\begin{lemma}\label{lem:key3}
Consider a definably complete locally o-minimal structure $\mathcal M=(M,<,\ldots)$ enjoying the properties (b) and (c) in Definition \ref{def:tame_top2}.
Let $X$ be a definable subset of $M^{m+n}$.
Set 
\[
S=\{x \in M^m\;|\; \text{the fiber }X_x \text{ has a nonempty interior}\}\text{.}
\]
If $S$ has a nonempty interior, $X$ also has a nonempty interior. 
\end{lemma}
\begin{proof}
We first consider the case in which $n=1$.
Take $c \in M$.
Consider the definable sets 
\begin{align*}
&X_{>c}=\{(x,y) \in M^m \times M\;|\; (x,y) \in X,\ y>c\}\text{ and }\\ 
&X_{<c}=\{(x,y) \in M^m \times M\;|\; (x,y) \in X,\ y<c\}\text{.}
\end{align*}
For any $x \in S$, at least one of the fibers $(X_{>c})_x$ and $(X_{<c})_x$ of $X_{>c}$ and $X_{<c}$ at $x$ has a nonempty interior by the property (b).
Set 
\begin{align*}
&S_{>c}=\{x \in M^m\;|\; \text{the fiber }(X_{>c})_x\text{ has a nonempty interior}\}\text{.}
\end{align*}
We define $S_{<c}$ in the same manner.
We get $S=S_{>c} \cup S_{<c}$.
Assume that $S$ has a nonempty interior.
At least one of $S_{>c}$ and $S_{<c}$ has a nonempty interior by the property (b) again.
We consider the case in which $S_{>c}$ has a nonempty interior.
We can prove the lemma similarly in the other case.
If $X_{>c}$ has a nonempty interior, the definable set $X$ obviously has a nonempty interior.
It implies that the lemma holds true for $X$ if it holds true for $X_{>c}$.
Therefore, we may assume that there is $c \in M$ satisfying $y>c$ for all $(x,y) \in X$.

Consider the definable function $f:S \rightarrow M$ given by $$f(x)=\inf\{y \in M\;|\; y \text{ is contained in the interior of the fiber }X_x\}\text{.}$$
It is well-defined by the above assumption.
Define the definable function $g:S \rightarrow M \cup \{\infty\}$ by $$g(x)=\sup\{y \in M\;|\; X_x \text{ contains an interval }(f(x),y)\}\text{.}$$
There exists an open box $V$ contained in $S$ such that the restrictions of $f$ and $g$ to $V$ are continuous by the properties (b) and (c) in Definition \ref{def:tame_top2}.
The set $X$ contains an open set $\{(x,y) \in V \times M\;|\; f(x)<y<g(x)\}$.
We have demonstrated the lemma for $n=1$.

We next consider the case in which $n>1$.
Consider the projection $\pi_1:M^{m+n} \rightarrow M^{m+n-1}$ forgetting the last coordinate and the projection $\pi_2:M^{m+n-1} \rightarrow M^m$ onto the first $m$ coordinates.
Set $\pi=\pi_2 \circ \pi_1$, 
\begin{align*}
T&=\{t \in \pi_1(X)\;|\;  \text{the fiber }X_t \text{ contains a nonempty open interval}\}\text{ and }\\
U &=\{u \in \pi(X)\;|\;  \text{the fiber }T_u \text{ has a nonempty interior}\}\text{.}
\end{align*}
The definable set $S$ is contained in $U$.
In particular, $U$ has a nonempty interior.
Applying the lemma to the pair of $T$ and the restriction of $\pi_2$ to $T$, we have $\myint(T) \not= \emptyset$ by the induction hypothesis.
We get $\myint(X) \not= \emptyset$ by the lemma for $n=1$.
\end{proof}

We do not use the following proposition in this paper, but it is worth to be mentioned.
It is a stronger version of definable Baire property discussed in \cite{FS,H}.
\begin{proposition}[Strong definable Baire property]\label{prop:baire}
Consider a definably complete locally o-minimal structure $\mathcal M=(M,<,\ldots)$ enjoying the property (a) in Definition \ref{def:tame_top}.
Take $c \in M$.
Let $\{X\langle r\rangle\}_{r>c}$ be a parameterized increasing family of definable sets of $M^n$; that is, there exists a definable subset $\mathcal X$ of $M^{n+1}$ such that $X\langle r\rangle$ coincides with the fiber $\mathcal X_r$ for any $r>c$ and we have $X\langle r\rangle \subset X\langle s\rangle$ if $r<s$.
Set $X=\bigcup_{r>c}X\langle r\rangle$.
The definable set $X\langle r\rangle$ has a nonempty interior for some $r>c$ if $X$ has a nonempty interior.  
\end{proposition}
\begin{proof}
The properties (b) and (c) in Definition \ref{def:tame_top2} follow from the property (a) by Theorem \ref{thm:dependence}.
We use this fact.

We prove the proposition by induction on $n$.
We first consider the case in which $n=1$.
Assume that $X\langle r\rangle$ have empty interiors for all $r>c$.
They are closed and discrete by Lemma \ref{lem:local}.
Set $Y=\{(r,x) \in M^2\;|\; r=\inf\{s \in M\;|\; x \in X\langle s\rangle\}\}$.
The set $Y$ is discrete.
In fact, consider the fiber $Y_r$ of $Y$ at $r$.
Take $r' \in M$ with $r'>r$.
We have $Y_r \subset X\langle r'\rangle$ because $\{X\langle r\rangle\}_{r>c}$ is a parameterized increasing family.
For any $x \in M$, there exists an open interval $I$ containing the point $x$ such that $X\langle r'\rangle \cap I$ consists of at most one point because $X\langle r'\rangle$ is discrete and closed.
Since $Y_r \subset X\langle r'\rangle$ whenever $r<r'$, the intersection $Y \cap ((c,r') \times I)$ consists of at most one point.
We have shown that $Y$ is discrete.
Since $X$ is the projection image of $Y$, $X$ is also discrete by the property (a).
We have demonstrated that $X$ has an empty interior.

We next consider the case in which $n>1$.
Assume that $X$ has a nonempty interior.
An open box $B$ is contained in $X$.
We may assume that $X=B$ considering $X\langle r\rangle \cap  B$ instead of $X\langle r\rangle$.
We lead to a contradiction assuming that $X\langle r\rangle$ have empty interiors for all $r>c$.
Take an open box $B_1$ in $M^{n-1}$ and an open interval $I_1$ with $B=B_1 \times I_1$.
Set $Y\langle r\rangle=\{x \in B_1\;|\; (X\langle r\rangle)_x \text{ contains an open interval}\}$.
The set $Y\langle r\rangle$ has an empty interior by Lemma \ref{lem:key3}.
We have $B_1 \not= \bigcup_{r>c}Y\langle r\rangle$ by the induction hypothesis.
Take $x \in B_1 \setminus  \left(\bigcup_{r>c}Y\langle r\rangle\right)$.
The union $\bigcup_{r>c}(X\langle r\rangle)_x$ has an empty interior because the fibers $(X\langle r\rangle)_x$ have empty interiors.
It contradicts the equality $\bigcup_{r>c}(X\langle r\rangle)_x=I_1$.
\end{proof}

\subsection{Dependence between the properties and local monotonicity property}
The satisfaction of the property (c) in Definition \ref{def:tame_top2} is related to local monotonicity property. Two local monotonicity properties are known.
Let $\mathcal M=(M,\ldots)$ be a locally o-minimal structure.
The first one is the \textit{weak local monotonicity property} given below.
\begin{quotation}
Let $I$ be an interval and $f:I \rightarrow M$ be a definable function.
For any $(a,b) \in M^2$, there exist an open interval $J_1$ containing the point $a$, an open interval $J_2$ containing the point $b$ and a mutually disjoint definable partition $f^{-1}(J_2) \cap J_1=X_d \cup X_c \cup X_+ \cup X_-$ satisfying the following conditions:
\begin{enumerate}
\item[(1)] the definable set $X_d$ is discrete and closed;
\item[(2)] the definable set $X_c$ is open and $f$ is locally constant on $X_c$;
\item[(3)] the definable set $X_+$ is open and $f$ is locally strictly increasing and continuous on $X_+$;
\item[(4)] the definable set $X_-$ is open and $f$ is locally strictly decreasing and continuous on $X_-$.
\end{enumerate}
\end{quotation}
The \textit{strong local monotonicity property} is the same as the weak one except that we can take $J_1=I$ and $J_2=M$.

The weak local monotonicity property is possessed by strongly locally o-minimal structures \cite[Proposition 11]{KTTT} and by uniformly locally o-minimal structures of the second kind \cite[Corollary 3.1]{Fuji}.
A model of DCTC enjoys the strong local monotonicity property \cite[Theorem 3.2]{S}.
On the other hand, the strongly locally o-minimal structure given in \cite[Example 12]{KTTT} is not definably complete, and has neither the property (a)  in Definition \ref{def:tame_top}, the property (c) in Definition \ref{def:tame_top2} nor strong local monotonicity property.

We discuss on the dependence between the properties in Definition \ref{def:tame_top} and Definition \ref{def:tame_top2}.
We use the following technical definition in the proof.
\begin{definition}
Consider an expansion of densely linearly ordered structure without endpoints $\mathcal M=(M,<,\ldots)$.
Let $A$ be a definable subset of $M^m$ and $f:A \rightarrow M$ be a definable function.
Let $1 \leq i \leq m$.
The function $f$ is \textit{$i$-constant} if, for any $a_1, \ldots, a_{i-1},a_{i+1},\ldots, a_n \in M$, the univariate function $f(a_1,\ldots, a_{i-1}, x, a_{i+1},\ldots, a_n)$ is constant.
We define that the function is \textit{$i$-strictly increasing} and \textit{$i$-strictly decreasing} in the same way.
The function is \textit{$i$-strictly monotone} if it is $i$-constant, $i$-strictly increasing or $i$-strictly decreasing.
The function $f$ is \textit{$i$-continuous} if, for any $a_1, \ldots, a_{i-1},a_{i+1},\ldots, a_n \in M$, the univariate function $f(a_1,\ldots, a_{i-1}, x, a_{i+1},\ldots, a_n)$ is continuous.
\end{definition}

In the proof of the theorem, the claim that the structure in consideration possesses the property (a) is simply called the property (a).
\begin{theorem}\label{thm:dependence}
Consider a definably complete locally o-minimal structure $\mathcal M=(M,<,\ldots)$.
\begin{enumerate}
\item[(i)] The property (c) in Definition \ref{def:tame_top2} implies the property (b).
\item[(ii)] The property (a) in Definition \ref{def:tame_top} implies the strong local monotonicity property and the property (d).
\item[(iii)] The strong local monotonicity property implies the properties (b) and (c).
\item[(iv)] The properties (c) and (d) imply the property (a).
\end{enumerate}
\end{theorem}
\begin{proof}
(i) We can prove it in the same manner as the proof of \cite[Theorem 3.3]{Fuji} using definable completeness instead of uniform local o-minimality of the second kind.
We omit the proof.

(ii) We can prove that the property (a) implies the strong local monotonicity property in the same manner as the proof of \cite[Theorem 3.2]{S} using the property (a) instead of \cite[Lemma 3.1i]{S}.
We omit the proof.

We next demonstrate the property (d).
We first demonstrate that the property (d) holds true when $n=d+1$.
We may assume that $\pi$ is the projection forgetting the first coordinate without loss of generality.
Take an arbitrary element $c \in M$.
The function $\eta:\pi(X) \rightarrow M$ is given by $\eta(x) = \inf\{x \in X_x\;|\; x \geq c\}$ if the definable set $\{x \in X_x\;|\; x \geq c\}$ is not empty and given by $\eta(x) = \sup\{x \in X_x\;|\; x \leq c\}$ otherwise.
It is a well-defined definable function.
The definable function $\tau:\pi(X) \rightarrow M^n$ is given by $\tau(x)=(\eta(x),x)$.
By Lemma \ref{lem:key0}, the fiber $X_x$ is closed for any $x \in \pi(X)$ by the assumption.
Therefore, we have $\tau(\pi(X)) \subset X$.
We have constructed the desired map.

We next show that the property (d) holds true by induction on $m=n-d$.
We may assume that $\pi$  is the projection onto the last $d$ coordinates without loss of generality.
 We have proven the case in which $m=1$.
 Consider the case in which $m>1$.
 Let $p:M^n \rightarrow M^{n-1}$ and $q:M^{n-1} \rightarrow M^d$ be the projection forgetting the first coordinate and the projection onto the last $d$ coordinates, respectively.
 We get $\pi=q \circ p$.
 The definable set $p(X) \cap q^{-1}(x)=p(\pi^{-1}(x) \cap X)$ is discrete by the property (a) for any $x \in \pi(X)$.
 Applying the induction hypothesis to $p$ and $q$, we can find definable maps $\tau_1: \pi(X) \rightarrow p(X)$ and $\tau_2:p(X) \rightarrow X$ such that the compositions $q \circ \tau_1$ and $p \circ \tau_2$ are identity maps.
 The composition $\tau=\tau_2 \circ \tau_1$ is the desired map.
 
 (iii) We demonstrate the properties (b) and (c) by induction on $m$ simultaneously.
The former is obvious because the structure $\mathcal M$ is locally o-minimal.
The property (c) follows from the strong local monotonicity property.

We consider the case in which $m>1$.
We first prove the property (b).
Assume that $X$ has a nonempty interior.
Take a bounded open box $B$ contained in $X$.
We may assume that $X=B$ considering $X_1 \cap B$ and $X_2 \cap B$ instead of $X_1$ and $X_2$, respectively.
We have $B=B_1 \times I_1$ for some open interval $I_1$ and an open box $B_1$ in $M^{m-1}$.
Set $Y_i=\{x \in B_1\;|\; \text{the fiber } (X_i)_x \text{ contains an open interval}\}$ for $i=1,2$.
Applying the property (b) in the case of $m=1$, we obtain $B_1 = Y_1 \cup Y_2$.
Applying the property (b) for $m-1$ to $B_1 = Y_1 \cup Y_2$, $Y_1$ or $Y_2$ has a nonempty interior.
We may assume that $\myint(Y_1) \not=\emptyset$ without loss of generality.
We may further assume that $Y_1=B_1$ shrinking $B$ if necessary.

Consider the function $f:B_1 \rightarrow \overline{I_1}$ given by 
\begin{align*}
f(x) &= \inf\{y \in I_1\;|\; \exists \alpha \in (X_1)_x,\ \exists \beta \in (X_1)_x \text{ such that } \alpha < y< \beta\\
& \text{ and } \forall y' \text{ with } \alpha<y'<\beta \text{, we have }y' \in (X_1)_x\}\text{.}
\end{align*}
Since $(X_1)_x$ contains an open interval and $\mathcal M$ is definably complete, the function $f$ is well-defined.
We next define the function $g:B_1 \rightarrow \overline{I_1}$ by 
\begin{align*}
g(x) &= \sup\{y \in I_1\;|\; y>f(x)  \text{ and } \forall y' \text{ with } f(x)<y'<y \text{, we have }y' \in (X_1)_x\}\text{.}
\end{align*}
The function $g$ is also well-defined for the same reason.
We have $f(x)<g(x)$ for all $x \in B_1$.
Apply the property (c) for $m-1$ to $f$ and $g$.
There exists an open box $V$ such that the restrictions of $f$ and $g$ to $V$ are continuous.
The definable set $X_1$ contains the open set $\{(x,y) \in V \times M\;|\; f(x)<y<g(x)\}$.
We have proven the property (b).

We next demonstrate the property (c).
We can prove the property (c) for arbitrary $n$ by an easy induction on $n$ when the property (c) holds true for $n=1$.
We may assume that $n=1$.
We may further assume that the domain of definition of $f$ is a bounded open box $B$ without loss of generality.
We define $I_1$ and $B_1$ in the same way as above.
Set 
\begin{align*}
X_+ &= \{(x,x') \in I_1 \times B_1\;|\; \text{the univariate function } f(\cdot, x') \text{ is }\\
& \quad \text{strictly increasing and continuous on a neighborhood of }x\}\text{,}\\
X_- &= \{(x,x') \in I_1 \times B_1\;|\; \text{the univariate function } f(\cdot, x') \text{ is }\\
& \quad \text{strictly decreasing and continuous on a neighborhood of }x\}\text{,}\\
X_c &= \{(x,x') \in I_1 \times B_1\;|\; \text{the univariate function } f(\cdot, x') \text{ is }\\
& \quad \text{constant on a neighborhood of }x\}\text{ and }\\
X_p &= B \setminus (X_+ \cup X_- \cup X_c) \text{.}
\end{align*}
The fibers $(X_p)_x$ are discrete for all $x \in B_1$ by the strong local monotonicity property.
In particular, $X_p$ has an empty interior.
At least one of $X_+$, $X_-$ and $X_c$ has a nonempty interior by the property (b) we have just proven.
Therefore, we may assume that $f$ is $1$-strictly monotone and $1$-continuous by considering an open box contained in one of them instead of $B$.
Applying the same argument $(m-1)$-times, we may assume that $f$ is $i$-strictly monotone and $i$-continuous for all $1 \leq i \leq m$.
The function $f$ is continuous on $B$ by \cite[Lemma 3.2.16]{vdD}.
We have proven the property (c).
 
 (iv) Let $X$ be a discrete definable subset of $M^n$.
 Let $\pi:M^n \rightarrow M^d$ be a coordinate projection.
 We prove that $\pi(X)$ is discrete.
 We first reduce to the case in which $d=1$.
 Assume that the claim is true for $d=1$.
 Take an arbitrary point $x \in \pi(X)$.
 Let $p_i:M^d \rightarrow M$ be the projection onto the $i$-th coordinate for $1 \leq i \leq d$.
 Since the composition $p_i \circ \pi$ is a coordinate projection, $p_i(\pi(X))$ is discrete.
 We can take an open interval $I_i$ such that $I_i \cap p_i(\pi(X)) = \{p_i(x)\}$.
 It is obvious that $\pi(X) \cap (I_1 \times \cdots \times I_d) = \{x\}$.
 It means that $\pi(X)$ is discrete.
 We have reduced to the case in which $d=1$.
 
 When $d=1$, there exists a definable map $\tau:\pi(X) \rightarrow X$ such that the composition $\pi \circ \tau$ is an identity map by the property (d).
 If $\pi(X)$ is not discrete, it contains an open interval $I$ because of local o-minimality.
Shrinking the interval $I$ if necessary, the restriction of $\tau$ to $I$ is continuous by the property (c).
It means that $X$ contains the graph of a continuous map defined on an open interval.
It contradicts the assumption that $X$ is discrete.
\end{proof}

\subsection{Uniformly locally o-minimal structure of the second kind}
We consider DCULOAS structures.
They were first introduced in \cite{Fuji} and their properties were also investigated in \cite{Fuji2}.
Their significant feature is that local definable cell decomposition for them is available.
We first review the definition of a locally o-minimal structure of the second kind.
\begin{definition}[\cite{Fuji}]
A locally o-minimal structure $\mathcal M=(M,<,\ldots)$ is a \textit{uniformly locally o-minimal structure of the second kind} if, for any positive integer $n$, any definable set $X \subset M^{n+1}$, $a \in M$ and $b \in M^n$, there exist an open interval $I$ containing the point $a$ and an open box $B$ containing $b$ such that the definable sets $X_y \cap I$ are finite unions of points and open intervals for all $y \in B$.
\end{definition}

We want to demonstrate that a DCULOAS structure enjoys the properties (a) through (d) in Definition \ref{def:tame_top} and Definition \ref{def:tame_top2}.

\begin{proposition}\label{prop:second}
A DCULOAS structure enjoys the properties (a) through (d) in Definition \ref{def:tame_top} and Definition \ref{def:tame_top2}.
\end{proposition}
\begin{proof}
We have only to demonstrate the property (a) by Theorem \ref{thm:dependence}.
We temporarily employ a definition of dimension different from Definition \ref{def:dim}.
The dimension considered here is that given in \cite[Definition 5.1]{Fuji}.
The above two definitions coincide by Theorem \ref{cor:alternative} once we obtain this proposition.

A discrete definable set is of dimension zero by \cite[Lemma 5.2]{Fuji}.
The projection image of the set of dimension zero is again of dimension zero by \cite[Theorem 1.1]{Fuji2}.
It is discrete by \cite[Corollary 5.3]{Fuji}.
We have demonstrated the property (a).
\end{proof}

\subsection{Model of DCTC}\label{sec:dctc}
Schoutens tried to figure out the common features of the models of the theory of all o-minimal structures \cite{S}.
A model of DCTC was introduced in his study.
He demonstrated tame topological properties enjoyed by it in \cite{S}.
The following is the definition of a model of DCTC.
\begin{definition}[\cite{S}]
A structure $\mathcal M=(M,<,\ldots)$ is a \textit{model of DCTC} if it is a definably complete expansion of a densely linearly ordered structure without endpoints with type completeness property.
A structure enjoys type completeness property by definition if the types $a^{+}$ and $a^{-}$ are complete for any $a \in M \cup \{\pm \infty\}$.
Here, a definable set $Y \subset M$ belongs to $a^+$ if there exists $b \in M$ with $b >a$ and $(a,b) \subset Y$.
We define $a^-$ similarly.
For instance, any definably complete locally o-minimal expansion of an ordered field, which is investigated in \cite{F}, is a model of DCTC.
\end{definition}

We demonstrate that a model of DCTC enjoys the properties in Definition \ref{def:tame_top} and Definition \ref{def:tame_top2}.

\begin{proposition}\label{prop:dctc}
A model of DCTC is a definably complete locally o-minimal structure enjoying the properties (a) through (d) in Definition \ref{def:tame_top} and Definition \ref{def:tame_top2}.
\end{proposition}
\begin{proof}
A model of DCTC is definably complete by the definition.
It is also a locally o-minimal structure by Lemma \ref{lem:local} and \cite[Proposition 2.6]{S}.
The property (a) is \cite[Corollary 4.3]{S}.
The properties (b) through (d) follow from Theorem \ref{thm:dependence}.
\end{proof}

\begin{corollary}\label{cor:dctc}
A definably complete locally o-minimal expansion of a field possesses the properties (a) through (d) in Definition \ref{def:tame_top} and Definition \ref{def:tame_top2}.
\end{corollary}
\begin{proof}
The corollary follows from Proposition \ref{prop:dctc} because a definably complete locally o-minimal expansion of a field is a model of DCTC.
\end{proof}

\section{Dimension theory}\label{sec:basic}

We develop a dimension theory for locally o-minimal structures possessing the property (a) in Definition \ref{def:tame_top}.
The properties (b) through (d) in Definition \ref{def:tame_top2} follow from the property (a) in Definition \ref{def:tame_top} by Theorem \ref{thm:dependence}.
We use this fact without notification in the rest of this paper.

\begin{definition}[Dimension]\label{def:dim}
Consider an expansion of a densely linearly order without endpoints $\mathcal M=(M,<,\ldots)$.
Let $X$ be a nonempty definable subset of $M^n$.
The dimension of $X$ is the maximal nonnegative integer $d$ such that $\pi(X)$ has a nonempty interior for some coordinate projection $\pi:M^n \rightarrow M^d$.
We set $\dim(X)=-\infty$ when $X$ is an empty set.
\end{definition}

A definable set of dimension zero is always closed and discrete.

\begin{proposition}\label{prop:zero}
Consider a locally o-minimal structure satisfying the property (a) in Definition \ref{def:tame_top}.
A definable set is of dimension zero if and only if it is discrete.
When it is of dimension zero, it is also closed.
\end{proposition}
\begin{proof}
Let $X$ be a definable subset of $M^n$.
The definable set $X$ is discrete if and only if the projection image $\pi(X)$ has an empty interior for all the coordinate projections $\pi:M^n \rightarrow M$ by the property (a).
Therefore, $X$ is discrete if and only if $\dim X=0$.
A discrete definable set is always closed by Lemma \ref{lem:key0}.
\end{proof}

The following two lemmas are key lemmas of this paper.
\begin{lemma}\label{lem:pre0}
Consider a definably complete locally o-minimal structure $\mathcal M=(M,<,\ldots)$ enjoying the properties (b) and (c) in Definition \ref{def:tame_top2}.
Let $X$ be a definable subset of $M^n$ of dimension $d$ and $\pi:M^n \rightarrow M^d$ be a coordinate projection such that the projection image $\pi(X)$ has a nonempty interior.
There exists a definable open subset $U$ of $M^d$ contained in $\pi(X)$ such that the fibers $X \cap \pi^{-1}(x)$ are discrete for all $x \in U$. 
\end{lemma}
\begin{proof}
Permuting the coordinates if necessary, we may assume that $\pi$ is the projection onto the first $d$ coordinates.
Set $$S=\{x \in \pi(X)\;|\; \text{the fiber }X \cap \pi^{-1}(x) \text{ is not discrete}\}\text{.}$$
We have $S=\{x \in \pi(X)\;|\; \dim(X \cap \pi^{-1}(x))>0\}$ by Proposition \ref{prop:zero}.
We want to show that $S$ has an empty interior.
Assume the contrary.
Let $\rho_j:M^n \rightarrow M$ be the coordinate projections onto the $j$-th coordinate for all $d<j \leq n$.
Set $$S_j=\{x \in \pi(X)\;|\; \rho_j(X \cap \pi^{-1}(x)) \text{ contains an open interval}\}\text{.}$$
We have $S=\bigcup_{d<j \leq n}S_j$ by the definition of dimension.
The definable set $S_j$ has a nonempty interior by the property (b) for some $d<j \leq n$.
Fix such $j$.
Let $\Pi:M^n \rightarrow M^{d+1}$ be the coordinate projection given by $\Pi(x)=(\pi(x),\rho_j(x))$.
The definable set $T=\{x \in M^d \;|\; \text{the fiber }(\Pi(X))_x \text{ contains an open interval}\}$ contains $S_j$ and it has a nonempty interior.
Therefore, the projection image $\Pi(X)$ has a nonempty interior by Lemma \ref{lem:key3}.
It contradicts the assumption that $\dim(X)=d$.
We have shown that $S$ has an empty interior.
Since $\pi(X)$ has a nonempty interior, there exists a definable open subset $U$ of $\pi(X)$ with $U \cap S=\emptyset$ by the property (b).
\end{proof}

\begin{lemma}\label{lem:pre1}
Consider a definably complete locally o-minimal structure $\mathcal M=(M,<,\ldots)$ enjoying the property (a) in Definition \ref{def:tame_top}.
Let $X \subset Y$ be definable subsets of $M^n$.
Assume that there exist a coordinate projection $\pi:M^n \rightarrow M^d$ and a definable open subset $U$ of $M^d$ contained in $\pi(X)$ such that the fibers $Y_x$ are discrete for all $x \in U$. 
Then, there exist
\begin{itemize}
\item a definable open subset $V$ of $U$,
\item a definable open subset $W$ of $M^n$ and 
\item a definable continuous map $f:V \rightarrow X$
\end{itemize}
such that 
\begin{itemize}
\item $\pi(W)=V$, 
\item $Y \cap W = f(V)$ and 
\item the composition $\pi \circ f$ is the identity map on $V$.  
\end{itemize}
\end{lemma}
\begin{proof}
Permuting the coordinates if necessary, we may assume that $\pi$ is the projection onto the first $d$ coordinates.
Let $\rho_j:M^n \rightarrow M$ be the coordinate projections onto the $j$-th coordinate for all $d<j \leq n$.
The fiber $Y_x$ is discrete for any $x \in U$ by the assumption.
Since $X_x$ is a definable subset of $Y_x$, $X_x$ is also a discrete set.
There exists a definable map $g:U \rightarrow X$ such that the composition $\pi \circ g$ is the identity map on $U$ by the property (d).
Note that $\rho_j(Y_x)$ is discrete and closed by the property (a) and Lemma \ref{lem:key0}.
Consider the definable functions $\kappa_j^+:U \rightarrow M \cup \{+\infty\}$ defined by 
\[
\kappa_j^+(x)=\left\{
\begin{array}{ll}
\inf \{t \in \rho_j(Y_x)\;|\; t>\rho_j(g(x))\} & \text{if } \{t \in \rho_j(Y_x)\;|\; t>\rho_j(g(x))\}\not=\emptyset \text{,}\\
+\infty & \text{otherwise}
\end{array}
\right.
\]
for all $d<j \leq n$.
We define $\kappa_j^-:U \rightarrow M \cup \{-\infty\}$ similarly.
Then, we have
\[
\pi^{-1}(x) \cap Y \cap (\{x\} \times (\kappa_{d+1}^-(x),\kappa_{d+1}^+(x)) \times \cdots \times (\kappa_{n}^-(x),\kappa_{n}^+(x))) = \{g(x)\}
\]
for all $x \in U$.
There exists a  definable open subset $V$ of $U$ such that the restriction $f$ of $g$ to $V$ and the restrictions of $\kappa_j^-$ and $\kappa_j^+$ to $V$ are all continuous by the properties (b) and (c).
Set $W=\{(x,y_{d+1},\ldots, y_n) \in V \times M^{n-d}\;|\; \kappa_j^-(x) < y_j < \kappa_j^+(x) \text{ for all } d<j \leq n\}$.
The definable sets $V$ and $W$ and a definable continuous map $f$ satisfy the requirements.
\end{proof}

Summarizing the above two lemmas, we get the following lemma.
\begin{lemma}\label{lem:pre}
Consider a definably complete locally o-minimal structure $\mathcal M=(M,<,\ldots)$ enjoying the property (a) in Definition \ref{def:tame_top}.
Let $X \subset Y$ be definable subsets of $M^n$ of dimension $d$.
There exist
\begin{itemize}
\item a coordinate projection $\pi:M^n \rightarrow M^d$,
\item a definable open subset $V$ of $\pi(X)$,
\item a definable open subset $W$ of $M^n$ and 
\item a definable continuous map $f:V \rightarrow X$
\end{itemize}
such that 
\begin{itemize}
\item $\pi(W)=V$, 
\item $Y \cap W = f(V)$ and 
\item the composition $\pi \circ f$ is the identity map on $V$.  
\end{itemize}
\end{lemma}
\begin{proof}
Immediate from the definition of dimension, Lemma \ref{lem:pre0} and Lemma \ref{lem:pre1}.
\end{proof}

We also need the following lemma and its corollary.
\begin{lemma}\label{lem:key4}
Let $\mathcal M=(M,<,\ldots)$ be as in Lemma \ref{lem:pre}.
Let $C \subset M^n$ be a definable open subset and $f:C \rightarrow M^n$ be a definable injective continuous map.
The image $f(C)$ has a nonempty interior.  
\end{lemma}
\begin{proof}
We may assume that $C$ is an open box without loss of generality.
The lemma is obvious when $n=0$.
We assume that $n>0$.
We lead to a contradiction assuming that $f(C)$ has an empty interior.
Set $d=\dim f(C)$.
We have $0 \leq d <n$.
When $d=0$, the set $f(C)$ is discrete by Proposition \ref{prop:zero}.
The image $f(C)$ is a singleton by \cite[Proposition 1.6]{M} because the open box $C$ is definably connected.
Contradiction to the assumption that $f$ is injective.

We next consider the case in which $d \not=0$.
Applying Lemma \ref{lem:pre}, we can take a coordinate projection $\pi:M^n \rightarrow M^d$ and a definable open set $W$ of $M^n$ such that the restriction of $\pi$ to $f(C) \cap W$ is injective and its image is a definable open set.
We may assume that the restriction of $\pi$ to $f(C)$ is injective by considering $f^{-1}(W)$ instead of $C$.
Since $f$ is injective and continuous by the assumption, the composition $\pi \circ f$ is also injective and continuous.

Take an open box $B$ contained in $C$.
Let $B_1$ and $B_2$ be the open boxes in $M^d$ and $M^{n-d}$ with $B=B_1 \times B_2$, respectively.
Take $c \in B_2$.
Consider the definable map $g:B_1 \rightarrow M^d$ given by $g(x)=\pi(f(x,c))$.
It is injective and continuous.
There exists an open box $D$ in $M^d$ with $D \subset g(B_1)$ by the induction hypothesis.
Take a point $x_0 \in B_1$ with $g(x_0) \in D$ and a point $c' \in B_2$ sufficiently close to $c$ with $c' \not=c$.
We have $\pi(f(x_0,c')) \in D$ because $\pi \circ f$ is continuous.
There exists a point $x_1 \in B_1$ with $\pi(f(x_0,c'))=g(x_1)=\pi(f(x_1,c))$ because $D \subset g(B_1)$.
It contradicts the fact that  $\pi \circ f$ is injective.
\end{proof}

\begin{corollary}\label{cor:key}
Let $\mathcal M=(M,<,\ldots)$ be as in Lemma \ref{lem:pre}.
Let $B$ and $C$ be open boxes in $M^m$ and $M^n$, respectively.
If there exists a definable continuous injective map from $B$ to $C$, we have $m \leq n$.
\end{corollary}
\begin{proof}
We lead to a contradiction assuming that $m>n$.
Take a definable continuous injective map $f:B \rightarrow C$ and $c \in M^{m-n}$.
Consider the definable map $g:B \rightarrow C \times M^{m-n}$ given by $g(x)=(f(x),c)$.
It is obviously continuous and injective.
The image $g(B)$ has a nonempty interior by Lemma \ref{lem:key4}.
Contradiction.
\end{proof}

The following theorem is one of the main theorems of this paper.
\begin{theorem}\label{thm:dim}
Consider a definably complete locally o-minimal structure $\mathcal M=(M,<,\ldots)$ enjoying the property (a) in Definition \ref{def:tame_top}.
The following assertions hold true:
\begin{enumerate}
\item[(1)] Let $X \subset Y$ be definable sets.
Then, the inequality $\dim(X) \leq \dim(Y)$ holds true.
\item[(2)] Let $\sigma$ be a permutation of the set $\{1,\ldots,n\}$.
The definable map $\overline{\sigma}:M^n \rightarrow M^n$ is defined by $\overline{\sigma}(x_1, \ldots, x_n) = (x_{\sigma(1)},\ldots, x_{\sigma(n)})$.
Then, we have $\dim(X)=\dim(\overline{\sigma}(X))$ for any definable subset $X$ of $M^n$.
\item[(3)] Let $X$ and $Y$ be definable sets.
We have $\dim(X \times Y) = \dim(X)+\dim(Y)$.
\item[(4)] Let $X$ and $Y$ be definable subsets of $M^n$.
We have 
\begin{align*}
\dim(X \cup Y)=\max\{\dim(X),\dim(Y)\}\text{.}
\end{align*}
\item[(5)] Let $f:X \rightarrow M^n$ be a definable map. 
We have $\dim(f(X)) \leq \dim X$.
\item[(6)] Let $f:X \rightarrow M^n$ be a definable map. 
The notation $\mathcal D(f)$ denotes the set of points at which the map $f$ is discontinuous. 
The inequality $\dim(\mathcal D(f)) < \dim X$ holds true.
\item[(7)] Let $X$ be a definable set.
The notation $\partial X$ denotes the frontier of $X$ defined by $\partial X = \overline{X} \setminus X$.
We have $\dim (\partial X) < \dim X$.
\end{enumerate}
\end{theorem}
\begin{proof}
The assertions (1) and (2) are obvious.
We omit the proofs.
\medskip

We demonstrate the assertion (3).
Assume that $X$ and $Y$ are definable subsets of $M^m$ and $M^n$, respectively.
Set $d=\dim(X)$, $e=\dim(Y)$ and $f=\dim(X \times Y)$.
We first show that $d+e \leq f$.
In fact, let $\pi:M^m \rightarrow M^d$ and $\rho:M^n \rightarrow M^e$ be coordinate projections such that both $\pi(X)$ and $\rho(Y)$ have nonempty interiors.
The definable set $(\pi \times \rho)(X \times Y)$ has a nonempty interior.
Therefore, we have $d+e \leq f$.
We show the opposite inequality.
Let $\Pi:M^{m+n} \rightarrow M^f$ be a coordinate projection with $\myint(\Pi(X \times Y)) \not=\emptyset$.
There exist coordinate projections $\pi_1:M^m \rightarrow M^{f_1}$ and $\pi_2:M^n \rightarrow M^{f_2}$ with $\Pi=\pi_1 \times \pi_2$.
In particular, we get $f=f_1+f_2$.
Since $\Pi(X \times Y)$ has a nonempty interior, there exist open boxes $C \subset M^{f_1}$ and $D \subset M^{f_2}$ with $C \times D \subset \Pi(X \times Y)$.
We get $C \subset \pi_1(X)$ and $D \subset \pi_2(Y)$.
Hence, we have $d \geq f_1$ and $e \geq f_2$.
We finally obtain $d+e \geq f_1+f_2 =f$.
\medskip

We next show the assertion (4).
The inequality $\dim(X \cup Y) \geq \max\{\dim(X),\dim(Y)\}$ is obvious by the assertion (1).
We show the opposite inequality.
Set $d=\dim(X \cup Y)$.
There exists a coordinate projection $\pi:M^n \rightarrow M^d$ such that $\pi(X \cup Y)$ has a nonempty interior by the definition of dimension.
At least one of $\pi(X)$ and $\pi(Y)$ has a nonempty interior by the property (b) because $\pi(X \cup Y)=\pi(X) \cup \pi(Y)$.
We may assume that $\pi(X)$ has a nonempty interior without loss of generality.
We have $d \leq \dim(X)$ by the definition of dimension.
We have demonstrated that $\dim(X \cup Y) \leq \max\{\dim(X),\dim(Y)\}$.
\medskip

The next target is the assertion (5).
Let $X$ be a definable subset of $M^m$.
The notation $\Gamma(f)$ denotes the graph of the map $f$.
We first demonstrate that $\dim(\Gamma(f))=\dim(X)$.
In fact, the inequality $\dim(X) \leq \dim(\Gamma(f))$ is obvious because $X$ is the projection image of $\Gamma(f)$.
Set $d=\dim(\Gamma(f))$ and $e=\dim(X)$.

Applying Lemma \ref{lem:pre} to the graph $\Gamma(f)$, we can take a coordinate projection $\pi:M^{m+n} \rightarrow M^d$, an open box $V$ contained in $\pi(\Gamma(f))$ and a definable continuous map $\tau:V \rightarrow \Gamma(f)$ such that the composition $\pi \circ \tau$ is the identity map on $V$.
In particular, the map $\tau$ is injective.

Let $\Pi:M^{m+n} \rightarrow M^m$ be the projection onto the first $m$-coordinate.
The restriction of $\Pi$ to the graph $\Gamma(f)$ is obviously injective.
Applying Lemma \ref{lem:pre} to the set $X$, we can take a coordinate projection $\rho:M^m \rightarrow M^e$ and a definable open subset $W$ of $M^m$ such that the restriction of $\rho$ to $W \cap X$ is injective.
The inverse image $(\Pi \circ \tau)^{-1}(W)$ contains an open box because $\Pi \circ \tau$ is continuous.
Replacing $V$ with the open box, we may assume that the restriction of $\rho$ to $\Pi(\tau(V))$ is injective.
We finally get the definable continuous injective map $\rho \circ \Pi \circ \tau: V \rightarrow M^e$.
We have $d \leq e$ by Corollary \ref{cor:key}.
We have shown that $\dim X = \dim \Gamma(f)$.

It is now obvious that $\dim f(X) \leq \dim \Gamma(f) = \dim X$ because $f(X)$ is the projection image of $\Gamma(f)$.
\medskip

We demonstrate the assertion (6).
Let $X$ be a definable subset of $M^m$.
We lead to a contradiction assuming that $d=\dim X=\dim \mathcal D(f)$.
By Lemma \ref{lem:pre}, there exist a coordinate projection $\pi:M^m \rightarrow M^d$, definable open subsets $V \subset \pi(\mathcal D(f))$ and $W \subset M^m$ and a definable continuous function $g:V \rightarrow \mathcal D(f)$ such that $\pi(W)=V$, $X \cap W=g(V)$ and $\pi \circ g$ is the identity map on $V$.
Shrinking $V$ and replacing $W$ with $W \cap \pi^{-1}(V)$ if necessary, we may assume that $f \circ g$ is continuous by the property (c).
Since $g$ is a definable homeomorphism onto its image, the function $f$ is continuous on $g(V)= X \cap W$.
On the other hand, $f$ is discontinuous everywhere on $X \cap W$ because $X \cap W$ is open in $X$ and $X \cap W =g(V)$ is contained in $\mathcal D(f)$.
Contradiction.
We have demonstrated the assertion (6).
\medskip

The remaining task is to demonstrate the assertion (7).
Take distinct elements $c,d \in M$.
Consider the definable function $f:\overline{X} \rightarrow M$ given by
\[
f(x)= \left\{
\begin{array}{ll}
c & \text{if } x \in X \text{ and }\\
d & \text{otherwise.} 
\end{array}
\right.
\]
It is obvious that $\mathcal D(f)$ contains $\partial X$.
The assertion (7) follows from the assertions (1) and (6).
\end{proof}

\begin{remark}
Theorem \ref{thm:dim} (1) through (3) hold true for any expansion of a densely linearly order without endpoints.
Theorem \ref{thm:dim} (4) is valid for any locally o-minimal structure with the property (b).
\end{remark}

A \textit{constructible} set is a finite boolean combination of open sets.
We get the following corollary:
\begin{corollary}\label{cor:constructible}
Consider a definably complete locally o-minimal structure enjoying the property (a) in Definition \ref{def:tame_top}.
Any definable set is constructible.
\end{corollary}
\begin{proof}
Let $X$ be a definable set of dimension $d$.
We prove that $X$ is constructible by induction on $d$.
When $d=0$, the definable set $X$ is discrete and closed by Proposition \ref{prop:zero}.
In particular, it is constructible.
When $d>0$, the frontier $\partial X$ is of dimension smaller than $d$ by Theorem \ref{thm:dim}(7).
It is constructible by the induction hypothesis.
Therefore, $X= \overline{X} \setminus \partial X$ is also constructible. 
\end{proof}

The following theorem gives an alternative definition of dimension.
The alternative definition is the same as the definition of dimension given in \cite[Definition 5.1]{Fuji}.
\begin{theorem}\label{cor:alternative}
Consider a definably complete locally o-minimal structure $\mathcal M=(M,<,\ldots)$ enjoying the property (a) in Definition \ref{def:tame_top}.
A definable set $X$ is of dimension $d$ if and only if the nonnegative integer $d$ is the maximum of nonnegative integers $e$ such that there exist an open box $B$ in $M^e$ and 
a definable injective continuous map $\varphi:B \rightarrow X$ homeomorphic onto its image. 
\end{theorem}
\begin{proof}
Let $d'$ be the maximum of nonnegative integers $e$ satisfying the condition given in the theorem.
We first demonstrate $d' \leq d$.
In fact, let $B$ be an open box contained in $M^{d'}$ and $\varphi:B \rightarrow X$ be a definable injective continuous map homeomorphic onto its image.
We have $\dim \varphi(B)=\dim B=d'$ by Theorem \ref{thm:dim}(5).
We get $d = \dim X \geq \dim(\varphi(B)) = d'$ by Theorem \ref{thm:dim}(1).

We next demonstrate $d \leq d'$.
Applying Lemma \ref{lem:pre} to the definable set $X$, we can get a coordinate projection $\pi:M^n \rightarrow M^d$, a definable open box $U$ in $\pi(X)$ and a definable continuous map $\tau:U \rightarrow X$ such that $\pi \circ \tau$ is the identity map on $U$.
In particular, $\tau$ is a definable continuous injective map homeomorphic onto its image.
Therefore, we have $d \leq d'$ by the definition of $d'$.  
\end{proof}

We get the following corollary:
\begin{corollary}\label{cor:local_dim}
Let $\mathcal M=(M,<,\ldots)$ be as in Theorem \ref{cor:alternative}.
Let $X$ be a definable subset of $R^n$.
There exists a point $x \in M^n$ such that we have $\dim(X \cap B)=\dim(X)$ for any open box $B$ containing the point $x$.
\end{corollary}
\begin{proof}
Set $d=\dim(X)$.
There exists an open box $U$ in $M^d$ and a definable continuous injective map $\varphi:U \rightarrow X$ homeomorphic onto its image by Theorem \ref{cor:alternative}.
Take an arbitrary point $t \in U$ and set $x=\varphi(t)$.
For any open box $B$ containing the point $x$, the inverse image $\varphi^{-1}(B)$ is a definable open set.
Take a open box $V$ with $t \in V \subset \varphi^{-1}(B)$.
The restriction $\varphi|_{V} V \rightarrow X \cap B$ is a definable continuous injective map homeomorphic onto its image.
Hence, we have $\dim(X \cap B) \geq d$ by Theorem \ref{cor:alternative}.
The opposite inequality follows from Theorem \ref{thm:dim}(1).
\end{proof}

We begin to demonstrate the addition property of dimension for definably complete locally o-minimal structures enjoying the property (a) in Definition \ref{def:tame_top}.
It is a counterpart of \cite[Chapter 4, Proposition 1.5]{vdD} in the o-minimal case, that of \cite[Theorem 4.2]{W} in the weakly o-minimal case and that of \cite[Lemma 5.4]{Fuji} in the case of local o-minimal structure admitting local definable cell decomposition.
We first treat a special case.
\begin{lemma}\label{lem:addition0}
Consider a definably complete locally o-minimal structure $\mathcal M=(M,<,\ldots)$ enjoying the property (a) in Definition \ref{def:tame_top}.
Let $\varphi:X \rightarrow Y$ be a definable surjective map whose fibers $\varphi^{-1}(y)$ are discrete for all $y \in Y$.
We have $\dim X = \dim Y$.  
\end{lemma}
\begin{proof}
Let $X$ and $Y$ be definable subsets of $M^m$ and $M^n$, respectively.
Set $d=\dim(X)$ and $e=\dim(Y)$.

We first assume that $\varphi$ is continuous.
We have $d \geq e$ by Theorem \ref{thm:dim}(5).
We demonstrate the opposite inequality.
We first reduce to the case in which $X$ is a definable open subset of $M^m$.
There exist a definable open subset $U$ of $R^d$ and a definable continuous injective map $\sigma:U \rightarrow X$ homeomorphic onto its image by Theorem \ref{cor:alternative}.
If the lemma holds true for the composition $\varphi \circ \sigma$, we have $\dim X = d = \dim U = \dim \varphi \circ \sigma(U) \leq \dim Y=e$ by Theorem \ref{thm:dim}(1).
The lemma is also true for the original $\varphi$.
Hence, we may assume that $X$ is open in $M^m$.
In particular, we have $m=d$.

Let $\Pi:M^{m+n} \rightarrow M^n$ be the projection onto the last $n$ coordinates.
Consider the graph $\Gamma(\varphi)$ of $\varphi$.
Note that $\Pi^{-1}(y) \cap \Gamma(\varphi)$ are discrete for all $y \in Y$.
Take a coordinate projection $\pi:M^n \rightarrow M^e$ such that $\pi(Y)$ has a nonempty interior.
The definable set $(\pi \circ \Pi)^{-1}(z) \cap \Gamma(\varphi)$ is discrete and closed if $\pi^{-1}(z) \cap Y$ is discrete for $z \in M^e$ by Lemma \ref{lem:aaa}. 
By Lemma \ref{lem:pre0} and Lemma \ref{lem:pre1}, there exist definable open subsets $V \subset \pi(Y)$ and $W \subset M^{m+n}$ and a definable continuous map $\tau:V \rightarrow \Gamma(\varphi)$ such that $\pi \circ \Pi(W)=V$, $W \cap \Gamma(\varphi)=\tau(V)$ and $\pi \circ \Pi \circ \tau$ is the identity map on $V$.
In particular, the restriction of $\pi \circ \Pi$ to $W \cap \Gamma(\varphi)$ is injective.

Let $\iota:X \rightarrow \Gamma(\varphi)$ be the natural injection.
The map $\iota$ is continuous because $\varphi$ is continuous.
We may assume that $\pi \circ \Pi \circ \iota$ is injective replacing $X$ with an open box contained in the definable open set $\iota^{-1}(W)$.
We finally obtain the definable continuous injective map from an open box in $M^d$ to $M^e$.
We get $d \leq e$ by Corollary \ref{cor:key}.

We next demonstrate the lemma when $\varphi$ is not necessarily continuous by induction on $d$.
When $d=0$, the definable set $X$ is discrete and closed by Proposition \ref{prop:zero}.
In particular, the definable map $\varphi$ is continuous.
Therefore, the lemma holds true in this case.
We next consider the case in which $d>0$.
Let $\mathcal D(\varphi)$ be the set of points at which $\varphi$ is discontinuous.
We have $\dim \mathcal D(\varphi) < \dim X$ by Theorem \ref{thm:dim}(6).
We get $\dim \varphi(\mathcal D(\varphi))=\dim \mathcal D(\varphi)$ by the induction hypothesis.
We obtain $\dim (X \setminus \mathcal D(\varphi)) = \dim \varphi(X \setminus \mathcal D(\varphi))$ because $\varphi$ is continuous on $X \setminus \mathcal D(\varphi)$.
We finally get $\dim \varphi(X)= \max \{\dim \varphi(X \setminus \mathcal D(\varphi)), \dim \varphi(\mathcal D(\varphi))\}= \max \{\dim (X \setminus \mathcal D(\varphi)), \dim (\mathcal D(\varphi))\} = \dim(X)$ by Theorem \ref{thm:dim}(4).
\end{proof}

The following theorem is the second main theorem of this paper.

\begin{theorem}\label{thm:addition}
Consider a definably complete locally o-minimal structure $\mathcal M=(M,<,\ldots)$ enjoying the property (a) in Definition \ref{def:tame_top}.
Let $\varphi:X \rightarrow Y$ be a definable surjective map whose fibers are equi-dimensional; that is, the dimensions of the fibers $\varphi^{-1}(y)$ are constant.
We have $\dim X = \dim Y + \dim \varphi^{-1}(y)$ for all $y \in Y$.  
\end{theorem}
\begin{proof}
Let $X$ and $Y$ be definable subsets of $M^m$ and $M^n$, respectively.
Set $d=\dim(\varphi^{-1}(y))$, $e=\dim(Y)$ and $f=\dim(X)$.
We first reduce to the case in which there exists a coordinate projection $\pi:M^m \rightarrow M^d$ such that $\pi(\varphi^{-1}(y))$ have nonempty interiors for all $y \in Y$.
In fact, consider the set $\Pi_{m,d}$ of all the coordinate projections of $M^m$ onto $M^d$.
Set $Y_{\pi}=\{y \in Y\;|\; \pi(\varphi^{-1}(y)) \text{ has a nonempty interior}\}$.
We get $Y=\bigcup_{\pi \in \Pi_{m,d}} Y_\pi$ by the assumption.
Assume that the theorem is true for the restrictions of $\varphi$ to $\varphi^{-1}(Y_{\pi})$ for all $\pi \in \Pi_{m,d}$.
We have $$\dim X = \max_{\pi \in \Pi_{m,d}} \dim \varphi^{-1}(Y_{\pi}) = d + \max_{\pi \in \Pi_{m,d}} \dim Y_{\pi} = d + \dim(Y)$$ by Theorem \ref{thm:dim}(4).
The theorem holds true for the original $\varphi$.
We may assume that there exists a coordinate projection $\pi:M^m \rightarrow M^d$ such that $\pi(\varphi^{-1}(y))$ have nonempty interiors for all $y \in Y$.
We fix such a $\pi$ through the proof.

We next show that $d+e \leq f$.
By Lemma \ref{lem:pre}, we can get a coordinate projection $p:M^n \rightarrow M^e$, a definable open subset $W$ of $M^e$ contained in $p(Y)$ and a definable continuous injective map $\tau:W \rightarrow Y$ which is homeomorphic onto its image such that $p \circ \tau$ is the identity map and $p^{-1}(w) \cap Y$ is discrete for any $w \in W$.
Consider the definable set
$$T=\{(w,v) \in W \times M^d\;|\;v \in \pi(\varphi^{-1}(\tau(w))) \text{ and } \pi^{-1}(v) \cap \varphi^{-1}(\tau(w)) \text{ is discrete}\} \text{.}$$
The fiber $T_w$ has a nonempty interior for any $w \in W$ by Lemma \ref{lem:pre0}.
Therefore, the set $T$ has a nonempty interior by Lemma \ref{lem:key3}.
In particular, we have $\dim(T)=d+e$.

Consider the definable subset $S=(p \times \pi)^{-1}(T) \cap \Gamma'(\varphi) \cap (\tau(W) \times M^m)$ of $M^m \times M^n$, where $\Gamma'(\varphi)$ denotes the reversed graph of the definable map $\varphi$ given by $\Gamma'(\varphi)=\{(y,x) \in Y \times X\;|\;y=\varphi(x)\}$.
It is obvious that $(p \times \pi)(S)=T$ and $S \cap (p \times \pi)^{-1}(w,v)$ are discrete for all $(w,v) \in T$.
Apply the property (d) to $S$ and the projection $p \times \pi$.
We can get a definable map $\psi':T \rightarrow S$ such that $(p \times \pi) \circ \psi'$ is the identity map on $T$.
Set $\psi=\pi \circ \psi':T \rightarrow X$.
It is obviously injective.
We have $d+e = \dim(T)=\dim \psi(T)\leq f$ by Lemma \ref{lem:addition0} and Theorem \ref{thm:dim}(1). 

We next demonstrate the opposite inequality $d+e \geq f$.
There exist a coordinate projection $q:M^m \rightarrow M^f$, a definable open subset $U$ of $M^f$ contained in $q(X)$ and a definable continuous injective map $\sigma:U \rightarrow X$ by Lemma \ref{lem:pre}.
The notation $\mathcal D(\varphi)$ denotes the set of points at which $\varphi$ is discontinuous.
Since $\dim \mathcal D(\varphi) < \dim X=f$ by Theorem \ref{thm:dim}(6), the projection image $q(\mathcal D(\varphi))$ has an empty interior.
The difference $U \setminus q(\mathcal D(\varphi))$ has a nonempty interior by the property (b).
Shrinking $U$ if necessary, we may assume that $\varphi$ is continuous on $\sigma(U)$.
Take a coordinate projection $p:M^n \rightarrow M^e$ and a definable set $W$ as in the proof of the inequality $d+e \leq f$.
Set $$Z=\{(v,w) \in M^d \times W\;|\;\pi^{-1}(v) \cap (p \circ \varphi)^{-1}(w) \text{ is discrete}\}\text{.}$$
We demonstrate that the set $Z$ has a nonempty interior.
Fix a point $w \in W$.
We have only to demonstrate that $Z_w=\{v \in M^d\;|\;\pi^{-1}(v) \cap (p \circ \varphi)^{-1}(w) \text{ is discrete}\}$ has a nonempty interior for any $w \in W$ by Lemma \ref{lem:key3}.

For any $z \in p^{-1}(w) \cap Y$, set $B(z)=\{v \in M^d\;|\;\pi^{-1}(v) \cap \varphi^{-1}(z) \text{ is not discrete}\}$.
We have $\dim B(z)<d$ for any $z \in p^{-1}(w)$ by Lemma \ref{lem:key3}.
Consider the set $D=\{(v,z) \in M^d \times M^n\;|\;v \in B(z) \text{ and } z \in p^{-1}(w) \cap Y\}$.
We get $\dim D =\sup_{z \in p^{-1}(w) \cap Y} \dim B(z)<d$ by Theorem \ref{cor:alternative} because $p^{-1}(w) \cap Y$ is discrete.
The definable set $\bigcup_{z \in p^{-1}(w) \cap Y} B(z)$ is the projection image of $D$, and it is of dimension smaller than $d$ by Theorem \ref{thm:dim}(5).
In particular, it has an empty interior.
Consider the definable set $ Z'_w = \bigcup_{z \in p^{-1}(w) \cap Y} \pi(\varphi^{-1}(z)) \setminus \left( \bigcup_{z \in p^{-1}(w) \cap Y} B(z) \right)$.
The set $Z'_w$ has a nonempty interior by the property (b) because $\bigcup_{z \in p^{-1}(w) \cap Y} \pi(\varphi^{-1}(z))$ has a nonempty interior by the definition of $\pi$.
On the other hand, the set $Z_w$ contains the set $Z'_w$.
In fact, take a point $v \in Z'_w$.
Consider the restriction of $\varphi$ to $\pi^{-1}(v) \cap  (p \circ \varphi)^{-1}(w)$.
The image is contained in $p^{-1}(w) \cap Y$, and it is discrete.
The fiber at $z \in p^{-1}(w) \cap Y$ is $\pi^{-1}(v) \cap \varphi^{-1}(z)$ and it is also discrete by the definition of $B(z)$ and $Z'_w$.
Finally, the definable set $\pi^{-1}(v) \cap  (p \circ \varphi)^{-1}(w)$ is discrete by applying Lemma \ref{lem:aaa} to the restriction of $\varphi$.
We have demonstrated that $Z_w$ has a nonempty interior.
Therefore, the definable set $Z$ has a nonempty interior.

Take an open box $V$ contained in $Z$.
Consider the definable continuous map $\Phi:U \rightarrow M^d \times M^e$ given by $\Phi(x)=(\pi \circ \sigma(x), p \circ \varphi \circ \sigma(x))$.
Replacing the open definable set $U$ with the definable open set $\Phi^{-1}(V)$ if necessary, we may assume that $\Phi(U) \subset Z$.
By the definition of $Z$, the fiber $\Phi^{-1}(v,w)$ is discrete for any $(v,w) \in Z$.
Therefore, we have $f=\dim U = \dim(\Phi(U)) \leq d+e$ by Lemma \ref{lem:addition0} and Theorem \ref{thm:dim}(1). 
We have finished the proof of the theorem.
\end{proof}

The following corollary is the addition property theorem for definably complete locally o-minimal structures enjoying the property (a) in Definition \ref{def:tame_top}.
\begin{corollary}[Addition property]
Let $\mathcal M=(M,<,\ldots)$ be as in Theorem \ref{thm:addition}.
Let $X$ be a definable subset of $M^m \times M^n$.
Set $X(d)=\{x \in M^m\;|\; \dim X_x=d\}$ for any nonnegative integer $d$.
The set $X(d)$ is definable and we have 
$$\dim \left( \bigcup_{x \in X(d)} \{x\} \times X_x\right)=\dim X(d) +d\text{.}$$
\end{corollary}
\begin{proof}
It is easy to prove that $X(d)$ is definable.
We omit the proof.
Apply Theorem \ref{thm:addition} to the restriction of the projection $\Pi:M^{m+n} \rightarrow M^m$ to the set $\bigcup_{x \in X(d)} \{x\} \times X_x$, then we get the corollary.
\end{proof}

The following corollary  also holds true:
\begin{corollary}
Let $\mathcal M=(M,<,\ldots)$ be as in Theorem \ref{thm:addition}.
Let $X$ be a definable subset of $M^{m+n}$ and $\pi:M^{m+n} \rightarrow M^m$ be a coordinate projection.
Fix a nonnegative integer $d$.
Assume that, for any $x \in  M^{m+n}$, there exists an open box $U$ containing the point $x$ satisfying the inequality $\dim (\pi(X \cap U)) \leq d$.
Then, we have $\dim(\pi(X)) \leq d$.
\end{corollary}
\begin{proof}
We first reduce to the case in which the fibers $X \cap \pi^{-1}(x)$ are equi-dimensional for all $x \in \pi(X)$.
In fact, set $Y_k=\{x \in \pi(X)\;|\; \dim(X \cap \pi^{-1}(x))=k\}$ and $X_k=X \cap \pi^{-1}(Y_k)$ for all $1 \leq k \leq n$.
They are definable because of the definition of dimension.
Since we have $\dim(\pi(X_k \cap U)) \leq \dim (\pi(X \cap U))$ for any open box $U$ by Theorem \ref{thm:dim}(1), the conditions in the corollary are satisfied for $X_k$.
Assume that the corollary holds true for $X_k$.
We have $\dim(Y_k)=\dim\pi(X_k) \leq d$.
We obtain $\dim(\pi(X))=\max_{1\leq k \leq n}\dim(Y_k) \leq d$ by Theorem \ref{thm:dim}(4).
The corollary is also true for $X$.
We have succeeded in reducing to the case in which the fibers are equi-dimensional.

Set $e=\dim (\pi(X))$ and $f=\dim(X \cap \pi^{-1}(x))$ for $x \in \pi(X)$.
We have $\dim(X)=e+f$ by Theorem \ref{thm:addition}.
We can take a point $b$ in $R^{m+n}$ such that $\dim(X \cap V)=e+f$ for any open box $V$ containing the point $b$ by Corollary \ref{cor:local_dim}.
Choose an open box $U$ containing the point $b$ so that $\dim (\pi(X \cap U)) \leq d$, which exists by the assumption.
Set $X'=X \cap U$.
It is obvious that the fibers $X' \cap \pi^{-1}(x)$ are of dimension not greater than $f$ for all $x \in \pi(X \cap U)=\pi(X')$.
Set $Y'_k=\{x \in \pi(X')\;|\; \dim(X' \cap \pi^{-1}(x))=k\}$ and $X'_k=X' \cap \pi^{-1}(Y'_k)$ for $1 \leq k \leq f$.
Since we have $X'=\bigcup_{k=1}^f X'_k$, we get $\dim(X'_l)=\dim(X')=e+f$ for some $1 \leq l \leq f$ by Theorem \ref{thm:dim}(4).
Again by Theorem \ref{thm:addition} and Theorem \ref{thm:dim}(1), we get
$e+f = \dim \pi(X'_l)+l\leq \dim(\pi(X \cap U)) + l \leq d+l$.
We finally obtain $e \leq d$ because $0 \leq l \leq f$.
\end{proof}

\section{Decomposition into quasi-special submanifolds}\label{sec:decomposition}

A decomposition theorem into quasi-special submanifolds is discussed in this section.
We first define quasi-special submanifolds.
\begin{definition}\label{def:quasi-special}
Consider an expansion of a densely linearly order without endpoints $\mathcal M=(M,<,\ldots)$.
Let $X$ be a definable subset of $M^n$ and $\pi:M^n \rightarrow M^d$ be a coordinate projection.
A point $x \in X$ is \textit{($X,\pi$)-normal} if there exists an open box $B$ in $M^n$ containing the point $x$ such that $B \cap X$ is the graph of a continuous map defined on $\pi(B)$ after permuting the coordinates so that $\pi$ is the projection onto the first $d$ coordinates.

A definable subset is a \textit{$\pi$-quasi-special submanifold} or simply a \textit{quasi-special submanifold} if, $\pi(X)$ is a definable open set and, for every point $x \in \pi(X)$, there exists an open box $U$ in $M^d$ containing the point $x$ satisfying the following condition:
For any $y \in X \cap \pi^{-1}(x)$, there exist an open box $V$ in $M^n$ and a definable continuous map $\tau:U \rightarrow M^n$ such that $\pi(V)=U$, $\tau(U)=X \cap V$ and the composition $\pi \circ \tau$ is the identity map on $U$.

Let $\{X_i\}_{i=1}^m$ be a finite family of definable subsets of $M^n$.
A \textit{decomposition of $M^n$ into quasi-special submanifolds partitioning $\{X_i\}_{i=1}^m$} is a finite family of quasi-special submanifolds $\{C_i\}_{i=1}^N$ such that $\bigcup_{i=1}^NC_i =M^n$, $C_i \cap C_j=\emptyset$ when $i \not=j$ and $C_i$ has an empty intersection with $X_j$ or is contained in $X_j$ for any $1 \leq i \leq m$ and $1 \leq j \leq N$.
A decomposition $\{C_i\}_{i=1}^N$ of $M^n$ into quasi-special submanifolds \textit{satisfies the frontier condition} if the closure of any quasi-special manifold $\overline{C_i}$ is the union of a subfamily of the decomposition.
\end{definition}

The following lemma guarantees that a definable set $X$ in which all the points are ($X,\pi$)-normal is always a $\pi$-quasi-special submanifold.
This property makes the proof of the decomposition theorem easy.
\begin{lemma}\label{lem:normal}
Consider a definably complete locally o-minimal structure $\mathcal M=(M,<,\ldots)$ enjoying the property (a) in Definition \ref{def:tame_top}.
Let $X$ be a definable subset of $M^n$ and $\pi:M^n \rightarrow M^d$ be a coordinate projection.
Assume that all the points $x \in X$ are ($X,\pi$)-normal.
Then, $X$ is a $\pi$-quasi-special submanifold.
\end{lemma}
\begin{proof}
We may assume that $\pi$ is the projection onto the first $d$ coordinates without loss of generality.
It is obvious that $\pi(X)$ is open because $X$ is locally the graph of a continuous map.
We fix a point $c \in \pi(X)$.
Note that the fiber $X_c = X \cap \pi^{-1}(c)$ is discrete by the assumption.
The fiber $X_c$ is also closed by Lemma \ref{lem:key0}.
Let $p_e: M^d \rightarrow M^e$ be the projection onto the first $e$ coordinates for all $0 \leq e \leq d$.
We demonstrate the following claim.
The lemma is obvious from the claim for $e=d$.

\medskip
\textbf{Claim.}
Let $e$ be a nonnegative integer with $0 \leq e \leq d$.
There exists an open box $U_e$ in $M^e$ containing the point $p_e(c)$ such that, for any $y \in X_c$, there exist an open box $V_{e,y}$ in $M^{d-e}$ and an open box $W_{e,y}$ in $M^n$ such that $y \in W_{e,y}$, $\pi(W_{e,y})=U_e \times V_{e,y}$ and the intersection of $X$ with $W_{e,y}$ is the graph of a continuous map defined on $U_e \times V_{e,y}$.
\medskip

We prove the claim by induction on $e$.
The claim follows from the assumption that all the points $x \in X$ are ($X,\pi$)-normal when $e=0$.
Consider the case in which $e>0$.
Let $c_e$ be the $e$-th coordinate of the point $c$.
Take an element $d_{+,e} \in M$ with $c_e < d_{+,e}$.
For any $y \in X_c$, let $\varphi_+(y)$ be the supremum of the point $x' \in M$ satisfying 
\begin{enumerate}
\item[(i)] $c_e < x' < d_{+,e}$, and 
\item[(ii)] that there exist $a \in M$ with $a < c_e$, an open box $V'_{e,y}$ in $M^{d-e}$ and an open box $W'_{e,y}$ in $M^n$ such that 
\begin{itemize}
\item $y \in W'_{e,y}$, 
\item $\pi(W'_{e,y})=U_{e-1} \times (a,x') \times V'_{e,y}$ and
\item the intersection of $X$ with $W'_{e,y}$ is the graph of a continuous map defined on $\pi(W'_{e,y})$.
\end{itemize}
\end{enumerate}
Such $x'$ exists and the value $\varphi_+(y)$ is lager than $c_e$ by the induction hypothesis.
We get a definable function $\varphi_+:X_c \rightarrow M$.
The image $\varphi_+(X_c)$ is discrete by the property (a) because the fiber $X_c$ is discrete.
It is closed by Lemma \ref{lem:key0}.
Set $b'_{e,+}=\inf\{z \in \varphi_+(X_c)\}$.
We have $b'_{e,+} > c_e$ because $\varphi_+(X_c) > c_e$.
Take $b_{e,+}$ so that $c_e<b_{e,+}<b'_{e,+}$.

Take an element $d_{-,e} \in M$ with $c_e > d_{-,e}$.
For any $y \in X_c$, we define $\varphi_-(y)$ as the infimum of the point $x \in M$ satisfying 
\begin{enumerate}
\item[(i')] $c_e > x > d_{-,e}$, and 
\item[(ii')] that there exist an open box $V_{e,y}$ in $M^{d-e}$ and an open box $W_{e,y}$ in $M^n$ such that 
\begin{itemize}
\item $y \in W_{e,y}$, 
\item $\pi(W_{e,y})=U_{e-1} \times (x,b_{e,+}) \times V_{e,y}$ and
\item the intersection of $X$ with $W_{e,y}$ is the graph of a continuous map defined on $\pi(W_{e,y})$.
\end{itemize}
\end{enumerate}
We can take a point $x$ satisfying the above conditions (i') and (ii').
In fact, we can take $a$, $V'_{e,y}$ and $W'_{e,y}$ satisfying the condition (ii) by putting $x'=b_{e,+}$.
Let $x$ be an element satisfying the inequality $x \geq a$ and the condition (i'), then $V_{e,y}=V'_{e,y}$ and $W_{e,y}=W'_{e,y} \cap \pi^{-1}(U_{e-1} \times (x,b_{e,+}) \times V_{e,y})$ satisfy the condition (ii').

In the same way as above, the supremum $b'_{e,-}=\sup\{z \in \varphi_-(X_c)\}$ satisfies the inequality $b'_{e,-}<c_e$.
We take $b_{e,-}$ so that $b'_{e,-}<b_{e,-}<c_e$.
Set $U_e=U_{e-1} \times (b_{e,-}, b_{e,+})$.
It is now obvious that $U_e$ satisfies the claim.
We have finished the proofs of both the claim and the lemma.
\end{proof}

We next construct a decomposition of a single definable set.
\begin{lemma}\label{lem:decomposition}
Consider a definably complete locally o-minimal structure $\mathcal M=(M,<,\ldots)$ enjoying the property (a) in Definition \ref{def:tame_top}.
Let $X$ be a definable subset of $M^n$.
There exists a family $\{C_i\}_{i=1}^N$ of mutually disjoint quasi-special submanifolds with $X= \bigcup_{i=1}^N C_i$ and $N \leq 2^{n}$.
\end{lemma}
\begin{proof}
We first define the full dimension of a definable subset $X$ of $M^n$.
Set $d=\dim X$.
The notation $\Pi_{n,d}$ denotes the set of all the coordinate projections of $M^n$ onto $M^d$.
The set $\Pi_{n,d}$ is a finite set. 
The \textit{full dimension} $\myfdim(X)$ is $(d,e)$ by definition if $d=\dim(X)$ and $e$ is the number of elements in $\Pi_{n,d}$ under which the projection image of $X$ has a nonempty interior.
The pairs $(d,e)$ are ordered by the lexicographic order.

We prove the the theorem by induction on $\myfdim(X)$.
When $\dim(X)=0$, $X$ is closed and discrete by Proposition \ref{prop:zero}.
The definable set $X$ is obviously a quasi-special submanifold in this case.

We consider the case in which $\dim(X) >0$.
Set $(d,e)=\myfdim(X)$.
Take a coordinate projection $\pi:M^n \rightarrow M^d$ such that $\pi(X)$ has a nonempty interior.
Set $G=\{x \in X\;|\; x \text{ is } (X,\pi) \text{-normal}\}$ and $B=X \setminus G$.
It is obvious that any point $x \in G$ is $(G,\pi)$-normal.
The definable set $G$ is $\pi$-quasi-special submanifold by Lemma \ref{lem:normal}.

We demonstrate that $\pi(B)$ has an empty interior.
Assume the contrary.
There exists an open box $U$ such that the fibers $B_x=\pi^{-1}(x) \cap B$ are discrete for all $x \in U$ by Lemma \ref{lem:pre0}.
We can take a definable map $\tau:U \rightarrow B$ with $\pi(\tau(x))=x$ for all $x \in U$ because the structure $\mathcal M$ possesses the property (d) in Definition \ref{def:tame_top2}.
The dimension of points at which the map $\tau$ is discontinuous is of dimension smaller than $d$ by Theorem \ref{thm:dim}(6).
We may assume that the restriction of $\tau$ to $U$ is continuous shrinking $U$ if necessary.

Set $Z=\partial (X \setminus \tau(U))$.
We get $\dim Z = \dim \partial (X \setminus \tau(U)) < \dim (X \setminus \tau(U)) \leq \dim X=d$ by Theorem \ref{thm:dim}(1), (7).
We have $\dim \overline{Z} = \dim Z <d$ again by Theorem \ref{thm:dim}(4), (7). 
On the other hand, we have $d=\dim U =\dim \pi(\tau(U)) \leq \dim \tau(U) \leq \dim X =d$ by Theorem \ref{thm:dim}(1), (5). 
We get $\dim(\tau(U))=d$.
It means that $\tau(U) \not\subset \overline{Z}$ by Theorem \ref{thm:dim}(1).

Take a point $p$ in $\tau(U) \setminus \overline{Z}$.
Take a sufficiently small open box $V$ containing the point $p$. 
We have $X \cap V=\tau(U) \cap V$ by the definition of $Z$ and $p$.
Since the restriction of $\tau$ to $U$ is continuous, there exists an open box $U'$ contained in $U \cap \tau^{-1}(V)$.
Consider the open box $V'=V \cap \pi^{-1}(U')$.
It is obvious that $X \cap V'=\tau(U) \cap V'$ is the graph of the restriction of $\tau$ to $U'$ by the definition.
Any point $\tau(U) \cap V'$ is $(X,\pi)$-normal, but it contradicts to the definition of $B$ and the inclusion $\tau(U) \subset B$.
We have shown that $\pi(B)$ has an empty interior.
In particular, we get $\myfdim(B) < \myfdim(X)$.

There exists a decomposition $B=C_1 \cup \ldots \cup C_k$ of $B$ satisfying the conditions in the lemma by the induction hypothesis.
The decomposition $X=G \cup C_1 \cup \ldots \cup C_k$ is the desired decomposition of $X$.

It is obvious that the number of quasi-special submanifolds $N$ is not greater than $$\sum_{d=0}^n (\text{the cardinality of }\Pi_{n,d})=\sum_{d=0}^n \left(\begin{array}{c}n\\d\end{array}\right)=2^n\text{.}$$
\end{proof}

We finally get the following two decomposition theorems:
\begin{theorem}\label{thm:decomposition}
Consider a definably complete locally o-minimal structure $\mathcal M=(M,<,\ldots)$ enjoying the property (a) in Definition \ref{def:tame_top}.
Let $\{X_i\}_{i=1}^m$ be a finite family of definable subsets of $M^n$.
There exists a decomposition $\{C_i\}_{i=1}^N$ of $M^n$ into quasi-special submanifolds partitioning $\{X_i\}_{i=1}^m$ with $N \leq 2^{m+n}$.
\end{theorem}
\begin{proof}
Set $X_i^0=X_i$ and $X_i^1=M^n \setminus X_i$ for all $1 \leq i \leq m$.
For any $\sigma \in \{0,1\}^m$, the notation $\sigma(i)$ denotes the $i$-th component of $\sigma$.
Set $X_\sigma = \bigcap_{i=1}^m X_i^{\sigma(i)}$ for any $\sigma \in \{0,1\}^m$.
The family $\{X_\sigma\}_{\sigma \in \{0,1\}^m}$ is mutually disjoint and satisfies the equality $M^n=\bigcup_{\sigma \in \{0,1\}^m} X_\sigma$.
For all $\sigma \in \{0,1\}^m$, there exist families $\{C_{\sigma,j}\}_{j=1}^{N_\sigma}$ of mutually disjoint quasi-special submanifolds with $X_\sigma= \bigcup_{j=1}^{N_\sigma} C_{\sigma,j}$ and $N_\sigma \leq 2^n$ by Lemma \ref{lem:decomposition}.
The family $\bigcup_{\sigma  \in \{0,1\}^m} \{C_{\sigma,j}\}_{j=1}^{N_\sigma}$ gives the decomposition we are looking for.
\end{proof}

\begin{theorem}\label{thm:frontier_condition}
Consider a definably complete locally o-minimal structure $\mathcal M=(M,<,\ldots)$ enjoying the property (a) in Definition \ref{def:tame_top}.
Let $\{X_i\}_{i=1}^m$ be a finite family of definable subsets of $M^n$.
There exists a decomposition $\{C_i\}_{i=1}^N$ of $M^n$ into quasi-special submanifolds partitioning $\{X_i\}_{i=1}^m$ and satisfying the frontier condition.
Furthermore, the number $N$ of quasi-special submanifolds is not greater than the number uniquely determined only by $m$ and $n$.
\end{theorem}
\begin{proof}
By reverse induction on $d$, we construct a decomposition $\{C_{\lambda}\}_{\lambda \in \Lambda_d}$ of $M^n$ into quasi-special submanifolds partitioning $\{X_i\}_{i=1}^m$ such that the closures of all the quasi-special submanifolds of dimension not smaller than $d$ are the unions of subfamilies of the decomposition. 

When $d=n$, take a decomposition $\{D_\lambda\}_{\lambda \in \Lambda}$ of $M^n$ into quasi-special submanifolds partitioning $\{X_i\}_{i=1}^m$ by Theorem \ref{thm:decomposition}.
Set $\Lambda_n'=\{\lambda \in \Lambda\;|\;\dim(D_{\lambda})=n\}$.
Get a decomposition $\{E_\lambda\}_{\lambda \in \widetilde{\Lambda_n}}$ of $M^n$ into quasi-special submanifolds partitioning the family $\{D_\lambda\}_{\lambda \in \Lambda} \cup \{\overline{D_{\lambda}}\setminus D_{\lambda}\}_{\lambda \in \Lambda_n'}$.
Consider the set 
$$
\widetilde{\Lambda_n}'=\{\lambda \in \widetilde{\Lambda_n} \;|\; E_\lambda \text{ is not contained in any }D_{\lambda'}\text{ with }\lambda' \in \Lambda_n'\}\text{.}
$$
We always have $\dim(E_\lambda)<n$ for all $\lambda \in \widetilde{\Lambda_n}'$ by Theorem \ref{thm:dim}(7).
Hence, the family $\{D_\lambda\}_{\lambda \in \Lambda_n'} \cup \{E_\lambda\}_{\lambda \in \widetilde{\Lambda_n}'}$
is trivially a decomposition of $M^n$ into quasi-special submanifolds partitioning $\{X_i\}_{i=1}^m$ we are looking for.

We next consider the case in which $d < n$.
Let $\{D_{\lambda}\}_{\lambda \in \Lambda_{d+1}}$ be a decomposition of $M^n$ into quasi-special submanifolds partitioning $\{X_i\}_{i=1}^m$ such that the closures of all the quasi-special submanifolds of dimension not smaller than $d+1$ are the unions of subfamilies of the decomposition. 
It exists by the induction hypothesis.
Set $\Lambda_d'=\{\lambda \in \Lambda_{d+1}\;|\;\dim(D_{\lambda})=d\}$ and $\Lambda_d''=\{\lambda \in \Lambda_{d+1}\;|\;\dim(D_{\lambda}) \geq d\}$.
Get a decomposition $\{E^d_\lambda\}_{\lambda \in \widetilde{\Lambda_d}}$ of $M^n$ into quasi-special submanifolds partitioning the family $\{D_\lambda\}_{\lambda \in \Lambda_{d+1}} \cup \{\overline{D_{\lambda}}\setminus D_{\lambda}\}_{\lambda \in \Lambda_d'}$.
Set 
$\widetilde{\Lambda_d}'=\{\lambda \in \widetilde{\Lambda_d} \;|\; E_\lambda \text{ is not contained in any }D_{\lambda'}\text{ with }\lambda' \in \Lambda_d''\}\text{.}
$
The family $\{D_\lambda\}_{\lambda \in \Lambda_d''} \cup \{E_\lambda\}_{\lambda \in \widetilde{\Lambda_d}'}$
is a decomposition of $M^n$ into quasi-special submanifolds partitioning $\{X_i\}_{i=1}^m$ we want to construct.

The `furthermore' part of the theorem is obvious from the proof.
\end{proof}

\section*{Acknowledgment}
In the initial draft, the author considers structures simultaneously satisfying the property (a) in Definition \ref{def:tame_top} and the properties (b) through (d) in Definition \ref{def:tame_top2}, ignoring their dependence.
The author appreciates an anonymous referee for pointing out their dependence.

\end{document}